\documentclass[12pt, reqno]{amsart}
\UseRawInputEncoding
\usepackage{amsmath, amsthm, amscd, amsfonts, amssymb, graphicx, color}
\usepackage[bookmarksnumbered, colorlinks, plainpages]{hyperref}
\hypersetup{colorlinks=true,linkcolor=red, anchorcolor=green,
citecolor=cyan, urlcolor=red, filecolor=magenta, pdftoolbar=true}

\newtheorem{theorem}{Theorem}[section]
\newtheorem{lemma}[theorem]{Lemma}
\newtheorem{proposition}[theorem]{Proposition}
\newtheorem{corollary}[theorem]{Corollary}
\theoremstyle{definition}

\theoremstyle{remark}
\newtheorem{remark}[theorem]{Remark}
\numberwithin{equation}{section}

\begin{document}

\setcounter{page}{1}

\title[$w_{\rho}$-orthogonality and $w_{\rho}$-parallelism]
{Characterizations of $w_{\rho}$-Birkhoff--James orthogonality and $w_{\rho}$-parallelism}

\author[F.~Kittaneh \MakeLowercase{and} A.~Zamani]
{Fuad Kittaneh$^{1}$ \MakeLowercase{and} Ali Zamani$^{2}$}

\address{$^1$Department of Mathematics, The University of Jordan, Amman, Jordan}
\email{fkitt@ju.edu.jo}

\address{$^2$Department of Mathematics, Farhangian University, Tehran, Iran
\&
School of Mathematics and Computer Sciences, Damghan University, P.O.BOX 36715-364, Damghan, Iran}
\email{zamani.ali85@yahoo.com}


\begin{abstract}
We study the concepts of Birkhoff--James orthogonality and parallelism in Hilbert space operators, induced
by the operator radius norm $w_{\rho}(\cdot)$. In particular, we completely characterize
Birkhoff--James orthogonality and parallelism with respect to
$w_{\rho}(\cdot)$. As an application of the results presented, we obtain a well-known
characterization due to R.~Bhatia and P.~\v{S}emrl for the classical Birkhoff--James orthogonality
of Hilbert space operators. Some other related results are also discussed.
\end{abstract} \maketitle
\section{Introduction}\label{section.1}
Let $\big(\mathcal{H}, \langle \cdot, \cdot\rangle\big)$ be a complex Hilbert space equipped with the norm $\|\!\cdot\!\|$,
and let $\mathbf{S}_{\mathcal{H}}$ denote the unit ball of $\mathcal{H}$, i.e.,
$\mathbf{S}_{\mathcal{H}}=\big\{x\in \mathcal{H}: \,\|x\|\leq1\big\}$.
Let $\mathbb{B}(\mathcal{H})$ be the $C^{\ast}$-algebra of all bounded linear
operators on $\mathcal{H}$. For $\rho>0$ an operator $A\in\mathbb{B}(\mathcal{H})$ is
called a $\rho$-contraction (see \cite{Nagy.Foias.1966}) if there is a Hilbert space $\mathcal{K}(\supseteq \mathcal{H})$
and a unitary operator $U$ on $\mathcal{K}$
such that $A^nx = \rho PU^nx$ for all $x\in\mathcal{H}, n= 1, 2, \cdots$,
where $P$ is the orthogonal projection from $\mathcal{K}$ to $\mathcal{H}$.
Holbrook \cite{Holbrook.1968} and Williams \cite{Williams.1968} defined the operator
radii $w_{\rho}(\cdot)$ as the generalized Minkowski distance functionals
on $\mathbb{B}(\mathcal{H})$, i.e.,
\begin{align*}
w_{\rho}(A)=\inf\left\{t>0: \,t^{-1}A \,\,\mbox{is a $\rho$-contraction}\right\}.
\end{align*}
The operator radius $w_{\rho}(\cdot)$, usually referred to in the literature as the $\rho$-radius, plays a very important role in the study of unitary
$\rho$-dilations (see, e.g., \cite{Nagy.Foias.Bercovici.Kerchy.2010}).
It is well known that $w_{\rho}(A^*) =w_{\rho}(A)$ and $w_{\rho}(U^*AU) =w_{\rho}(A)$ for all $A$ and all unitary $U\in\mathbb{B}(\mathcal{H})$
i.e., $w_{\rho}(\cdot)$ is, respectively, self-adjoint and weakly unitarily invariant.
Moreover, $\rho$-radii have the properties:
\begin{align*}
w_1(A) = \|A\|,
\end{align*}
where $\|\!\cdot\!\|$ is the Hilbert space operator norm, that is, $\|A\|= \sup\big\{\|Ax\|:\,x\in\mathbf{S}_{\mathcal{H}}\big\}$ and
\begin{align*}
w_2(A) = w(A),
\end{align*}
where $w(\cdot)$ is the numerical radius, that is,
$w(A) = \sup\big\{|\langle Ax, x\rangle|:\,x\in\mathbf{S}_{\mathcal{H}}\big\}$.
For every $A\in\mathbb{B}(\mathcal{H})$, we also have
\begin{align}\label{I.12.4}
\frac{1}{\rho}\|A\|\leq w_{\rho}(A)\leq \|A\|.
\end{align}
If $A$ is normal (i.e., $A^*A=AA^*$), then
$w_{\rho}(A)=\begin{cases}
(2\rho^{-1}-1)\|A\| &\text{if\, $0<\rho<1$}\\
\|A\| &\text{if\, $\rho\geq1$}
\end{cases}$
and if $A$ is $2$-nilpotent (i.e., $A^2=0$), then $w_{\rho}(A)=\frac{1}{\rho}\|A\|$.
Notice that there is a major difference between
the case when $0 < \rho \leq 2$ and $2 < \rho < \infty$.
It is known that for $\rho\in (0, 2]$, $w_{\rho}(\cdot)$ is a norm on $\mathbb{B}(\mathcal{H})$
but for $\rho\in (2, \infty)$ is only a quasi-norm.
For proofs and more facts about the operator radii, we refer the reader to
\cite{Ando.Li.2010, Holbrook.1968, Holbrook.1971, Nagy.Foias.Bercovici.Kerchy.2010, Okubo.Ando.1975, Okubo.Ando.1976, Williams.1968}.

Now, let $\rho\in (0, 2]$ and $A, B\in \mathbb{B}(\mathcal{H})$.
We say that $A$ is $w_{\rho}$-Birkhoff--James orthogonal to $B$ (see \cite{Birkhoff, James}), in short $A\perp_{w_{\rho}}B$, if
\begin{align*}
w_{\rho}(A+\gamma B)\geq w_{\rho}(A) \quad \mbox{for all $\gamma \in \mathbb{C}$},
\end{align*}
or equivalently,
\begin{align*}
\displaystyle{\min_{\gamma \in \mathbb{C}}}\,w_{\rho}(A+\gamma B) = w_{\rho}(A).
\end{align*}
In particular, when $\rho = 1$ and $\rho = 2$, we obtain, respectively, the definitions of
the classical Birkhoff-–James orthogonality (written $A\perp B$)
and the numerical radius orthogonality (written $A\perp_{w}B$) in $\mathbb{B}(\mathcal{H})$.
The Birkhoff-–James orthogonality plays a very crucial role in the geometry
of Hilbert space operators, see \cite{A.G.K.R.Z.JMAA.2019, A.R.AFA.2014, B.S.LAA.1999,
Ch.Wo.2018, CH.Ki.JMAA.2021, E.M.P.2021, G.S.P.AOT.2017, Ke.JOP.2004, Ke.PAMS.2005,
Kim.Lee.LAA.1999, KO.Sa.TA.JMAA.2019, Magajna.LAA.2022, Pal.Roy.JMAA.2024, Poon.JMAA.2024, Roy.LAA.2022, Sin.2021, Turn.LAA.2017, Wo.2017, Wo.Zam.LAMA.2023}
and the references therein.

Characterizations of the Birkhoff-–James orthogonality for operators on various Banach spaces
and elements of an arbitrary Hilbert $C^*$-module were given in
\cite{A.R, B.G, E.A.Z.MED.2022, Gro.LAMA.2017, Gro.Sin.LAMA.2023, Mal.Paul.Sen.MM.2022, M.Z.CMB.2017, Sai.Mal.Paul.2018, Wo.OP.2016, ZAM.AFA.2019, Z.W.AFA.2020}.

Furthermore, we say that $A$ is $w_{\rho}$-parallel to $B$ (see \cite{Sed.LAA.2007}), and we write $A\parallel_{w_{\rho}}B$, if
\begin{align*}
w_{\rho}(A+\lambda B)= w_{\rho}(A) + w_{\rho}(B) \quad \mbox{for some $\lambda \in \mathbb{T}$},
\end{align*}
or equivalently,
\begin{align*}
\displaystyle{\max_{\lambda \in \mathbb{T}}}\,w_{\rho}(A+\lambda B) = w_{\rho}(A) + w_{\rho}(B).
\end{align*}
Here, as usual, $\mathbb{T}$ is the unit circle of the complex plane.
For $\rho = 1$ and $\rho = 2$, we also obtain, respectively, the definitions of
the the operator norm parallelism (written $A\parallel B$)
and the numerical radius parallelism (written $A\parallel_{w}B$) in $\mathbb{B}(\mathcal{H})$.
The concept of parallelism plays a significant role in the study of the geometric and the analytic properties of Banach space,
for instance, see \cite{BCMWZ-2019, M.S.P.BJMA.2019, Ray.BLMS.2022, Sed.LAA.2007, Z.M.I.M.2016, Z.M.C.N.LAMA.2019}.

Some other authors studied different aspects of parallelism
of bounded linear operators and elements of an arbitrary Hilbert $C^*$-module, see
\cite{Meh.Amy.Zam.BIMS.2020, Woj.IND.2017,
Zam.LAA.2016, Zam.P.2019, Z.M.C.M.B.2015, Z.M.I.M.2016}
and the references therein.

Motivated by these, in this paper we explore the $w_{\rho}$-Birkhoff--James orthogonality and the $w_{\rho}$-parallelism on $\mathbb{B}(\mathcal{H})$.
In particular, we present characterizations of the $w_{\rho}$-Birkhoff--James orthogonality and the $w_{\rho}$-parallelism.
Our results extend some previously known results which appeared in the literature \cite{B.S.LAA.1999, Mal.Paul.Sen.MM.2022, Meh.Amy.Zam.BIMS.2020, Z.M.C.M.B.2015}.
\section{The $w_{\rho}$-Birkhoff--James orthogonality}\label{section.2}
We start our work with the following proposition, which contains some basic properties of the relation $\perp_{w_{\rho}}$.
\begin{proposition}\label{P.001}
Let $A, B \in \mathbb{B}(\mathcal{H})$ and $\rho\in (0, 2]$. The following conditions are mutually equivalent:
\begin{itemize}
\item[(i)] $A\perp_{w_{\rho}}B$,
\item[(ii)] $A^*\perp_{w_{\rho}}B^*$,
\item[(iii)] $\alpha A\perp_{w_{\rho}}\beta B$ for all $\alpha, \beta\in \mathbb{C}\setminus\{0\}$,
\item[(iv)] $U^*AU\perp_{w_{\rho}}U^*BU$ for all unitary $U\in\mathbb{B}(\mathcal{H})$.
\end{itemize}
\end{proposition}
\begin{proof}
The proof immediately follows from the definition of the relation $\perp_{w_{\rho}}$ and the properties of $w_{\rho}(\cdot)$.
\end{proof}
\begin{remark}\label{R.001.1}
Let $A, B\in \mathbb{B}(\mathcal{H})$.

(i) If $\rho\in (0, 2]$ and $A$ is $2$-nilpotent, then $w_{\rho}(A)=\frac{1}{\rho}\|A\|$, and so the condition
$A\perp B$ implies $A\perp_{w_{\rho}}B$. Indeed, for every $\gamma\in \mathbb{C}$, by \eqref{I.12.4} it follows that
\begin{align*}
w_{\rho}(A+\gamma B)\geq \frac{1}{\rho}\|A+\gamma B\| \geq \frac{1}{\rho}\|A\| = w_{\rho}(A).
\end{align*}
(ii) If $\rho\in [1, 2]$ and $A$ is normal, then $w_{\rho}(A)=\|A\|$, and hence the condition
$A\perp_{w_{\rho}}B$ implies $A\perp B$. Indeed, for every $\gamma\in \mathbb{C}$, again by \eqref{I.12.4}, we have
\begin{align*}
\|A+\gamma B\|\geq w_{\rho}(A+\gamma B) \geq w_{\rho}(A)=\|A\|.
\end{align*}
\end{remark}
In order to prove our desired characterization of the $w_{\rho}$-Birkhoff--James orthogonality,
we need the following lemmas.
The first lemma has been proved recently in \cite[Theorem~3.1]{K.Z.LAA.2024}.
\begin{lemma}\label{L.001}
Let $X \in \mathbb{B}(\mathcal{H})$ and $\rho\in (0, 2]$. Then
\begin{align*}
w_{\rho}(X)= \frac{2}{\rho}\,w\left(\begin{bmatrix}
0 & \sqrt{\rho(2-\rho)}X \\
0 & (1-\rho)X
\end{bmatrix}\right).
\end{align*}
\end{lemma}
Our second lemma reads as follows. Our approach is similar to the one given in \cite[Theorem~1]{Z.W.AFA.2020}.
\begin{lemma}\label{L.002}
Let $A, B \in \mathbb{B}(\mathcal{H})$ and $\rho\in (0, 2]$. The following conditions are equivalent:
\begin{itemize}
\item[(i)] $w_{\rho}(A+rB)\geq w_{\rho}(A)$ for all $r\geq0$,
\item[(ii)] there exists a sequence $\left\{\begin{bmatrix}
x_n\\
y_n
\end{bmatrix}\right\}$ in $\mathbf{S}_{\mathcal{H}\oplus \mathcal{H}}$ such that
\begingroup\makeatletter\def\f@size{11}\check@mathfonts
\begin{align*}
\displaystyle{\lim_{n\rightarrow \infty}}\left|\left\langle Ay_n, \sqrt{\frac{8-4\rho}{\rho}}x_n+\frac{2-2\rho}{\rho}y_n\right\rangle\right|=w_{\rho}(A)
\end{align*}
\endgroup
and
\begingroup\makeatletter\def\f@size{10}\check@mathfonts
\begin{align*}
\displaystyle{\lim_{n\rightarrow \infty}}{\rm Re}\left(
\left\langle \sqrt{\rho(2-\rho)}x_n + (1-\rho)y_n, Ay_n\right\rangle
\left\langle By_n, \sqrt{\rho(2-\rho)}x_n + (1-\rho)y_n\right\rangle\right)\geq0.
\end{align*}
\endgroup
\end{itemize}
\end{lemma}
\begin{proof}
(i)$\Rightarrow$(ii) Let $w_{\rho}(A+rB)\geq w_{\rho}(A)$ for all $r\geq0$. We may assume that $w_{\rho}(A)\neq0$
otherwise (ii) trivially holds. So there exists $\varepsilon_0\in(0, 1)$ such that $w_{\rho}(A)\geq\varepsilon^2$ for all
$\varepsilon\in(0, \varepsilon_0)$. Therefore,
\begin{align}\label{I.1.L.002}
w_{\rho}(A+\varepsilon B)\geq w_{\rho}(A)\geq w_{\rho}(A)-\varepsilon^2\geq 0 \qquad \big(0<\varepsilon<\varepsilon_0\big).
\end{align}
Let $\varepsilon\in(0, \varepsilon_0)$. By Lemma \ref{L.001} we have
\begin{align*}
w\left(\begin{bmatrix}
0 & \sqrt{\rho(2-\rho)}(A+\varepsilon B)\\
0 & (1-\rho)(A+\varepsilon B)
\end{bmatrix}\right)=\frac{\rho}{2}\,w_{\rho}(A+\varepsilon B),
\end{align*}
and hence there exists a sequence $\left\{\begin{bmatrix}
x_n\\
y_n
\end{bmatrix}\right\}$ in $\mathbf{S}_{\mathcal{H}\oplus \mathcal{H}}$ such that
\begin{align}\label{I.2.L.002}
\displaystyle{\lim_{n\rightarrow \infty}}\left|\left\langle \begin{bmatrix}
0 & \sqrt{\rho(2-\rho)}(A+\varepsilon B)\\
0 & (1-\rho)(A+\varepsilon B)
\end{bmatrix}\begin{bmatrix}
x_n\\
y_n
\end{bmatrix}, \begin{bmatrix}
x_n\\
y_n
\end{bmatrix}\right\rangle\right|=\frac{\rho}{2}\,w_{\rho}(A+\varepsilon B).
\end{align}
Utilizing Lemma \ref{L.001}, \eqref{I.1.L.002} and \eqref{I.2.L.002}, we have
\begingroup\makeatletter\def\f@size{11}\check@mathfonts
\begin{align*}
\frac{\rho}{2}\,w_{\rho}(A) + \varepsilon\frac{\rho}{2}\,w_{\rho}(B)&= w\left(\begin{bmatrix}
0 & \sqrt{\rho(2-\rho)}A \\
0 & (1-\rho)A
\end{bmatrix}\right)+ \varepsilon w\left(\begin{bmatrix}
0 & \sqrt{\rho(2-\rho)}B \\
0 & (1-\rho)B
\end{bmatrix}\right)
\\& \geq
\displaystyle{\lim_{n\rightarrow \infty}}\left|\left\langle \begin{bmatrix}
0 & \sqrt{\rho(2-\rho)}A\\
0 & (1-\rho)A
\end{bmatrix}\begin{bmatrix}
x_n\\
y_n
\end{bmatrix}, \begin{bmatrix}
x_n\\
y_n
\end{bmatrix}\right\rangle\right|
\\& \qquad \qquad + \varepsilon\displaystyle{\lim_{n\rightarrow \infty}}\left|\left\langle \begin{bmatrix}
0 & \sqrt{\rho(2-\rho)}B\\
0 & (1-\rho)B
\end{bmatrix}\begin{bmatrix}
x_n\\
y_n
\end{bmatrix}, \begin{bmatrix}
x_n\\
y_n
\end{bmatrix}\right\rangle\right|
\\&\geq \displaystyle{\lim_{n\rightarrow \infty}}\left|\left\langle \begin{bmatrix}
0 & \sqrt{\rho(2-\rho)}(A+\varepsilon B)\\
0 & (1-\rho)(A+\varepsilon B)
\end{bmatrix}\begin{bmatrix}
x_n\\
y_n
\end{bmatrix}, \begin{bmatrix}
x_n\\
y_n
\end{bmatrix}\right\rangle\right|
\\& =\frac{\rho}{2}\,w_{\rho}(A+\varepsilon B)\geq \frac{\rho}{2}\,w_{\rho}(A),
\end{align*}
\endgroup
and so by letting $\varepsilon\rightarrow 0^+$ we obtain
\begingroup\makeatletter\def\f@size{11}\check@mathfonts
\begin{align*}
\displaystyle{\lim_{n\rightarrow \infty}}\left|\left\langle \begin{bmatrix}
0 & \sqrt{\rho(2-\rho)}A\\
0 & (1-\rho)A
\end{bmatrix}\begin{bmatrix}
x_n\\
y_n
\end{bmatrix}, \begin{bmatrix}
x_n\\
y_n
\end{bmatrix}\right\rangle\right|
=\frac{\rho}{2}\,w_{\rho}(A).
\end{align*}
\endgroup
This implies
\begingroup\makeatletter\def\f@size{11}\check@mathfonts
\begin{align*}
\displaystyle{\lim_{n\rightarrow \infty}}\left|\Big\langle \sqrt{\rho(2-\rho)}Ay_n, x_n\Big\rangle+ \Big\langle (1-\rho)Ay_n, y_n\Big\rangle\right|=\frac{\rho}{2}\,w_{\rho}(A),
\end{align*}
\endgroup
or equivalently,
\begingroup\makeatletter\def\f@size{11}\check@mathfonts
\begin{align*}
\displaystyle{\lim_{n\rightarrow \infty}}\left|\left\langle Ay_n, \sqrt{\frac{8-4\rho}{\rho}}x_n+\frac{2-2\rho}{\rho}y_n\right\rangle\right|=w_{\rho}(A).
\end{align*}
\endgroup
We also have
\begingroup\makeatletter\def\f@size{10}\check@mathfonts
\begin{align*}
&2\varepsilon\displaystyle{\lim_{n\rightarrow \infty}}{\rm Re}
\left(\left\langle \sqrt{\rho(2-\rho)}x_n + (1-\rho)y_n, Ay_n\right\rangle
\left\langle By_n, \sqrt{\rho(2-\rho)}x_n + (1-\rho)y_n\right\rangle\right)
\\& \qquad + \frac{\rho^2}{4}\,w^2_{\rho}(A) + \varepsilon^2\frac{\rho^2}{4}\,w_{\rho}(B)
\\&= 2\varepsilon\displaystyle{\lim_{n\rightarrow \infty}}{\rm Re}
\left(\left\langle \begin{bmatrix}
x_n\\
y_n
\end{bmatrix}, \begin{bmatrix}
0 & \sqrt{\rho(2-\rho)}A\\
0 & (1-\rho)A
\end{bmatrix}\begin{bmatrix}
x_n\\
y_n
\end{bmatrix}\right\rangle
\left\langle \begin{bmatrix}
0 & \sqrt{\rho(2-\rho)}B\\
0 & (1-\rho)B
\end{bmatrix}\begin{bmatrix}
x_n\\
y_n
\end{bmatrix}, \begin{bmatrix}
x_n\\
y_n
\end{bmatrix}\right\rangle
\right)
\\& \qquad +w^2\left(\begin{bmatrix}
0 & \sqrt{\rho(2-\rho)}A \\
0 & (1-\rho)A
\end{bmatrix}\right)+ \varepsilon^2 w^2\left(\begin{bmatrix}
0 & \sqrt{\rho(2-\rho)}B \\
0 & (1-\rho)B
\end{bmatrix}\right)
\\& \geq 2\varepsilon\displaystyle{\lim_{n\rightarrow \infty}}{\rm Re}
\left(\left\langle \begin{bmatrix}
x_n\\
y_n
\end{bmatrix}, \begin{bmatrix}
0 & \sqrt{\rho(2-\rho)}A\\
0 & (1-\rho)A
\end{bmatrix}\begin{bmatrix}
x_n\\
y_n
\end{bmatrix}\right\rangle
\left\langle \begin{bmatrix}
0 & \sqrt{\rho(2-\rho)}B\\
0 & (1-\rho)B
\end{bmatrix}\begin{bmatrix}
x_n\\
y_n
\end{bmatrix}, \begin{bmatrix}
x_n\\
y_n
\end{bmatrix}\right\rangle
\right)
\\& \quad +\displaystyle{\lim_{n\rightarrow \infty}}\left|\left\langle \begin{bmatrix}
0 & \sqrt{\rho(2-\rho)}A\\
0 & (1-\rho)A
\end{bmatrix}\begin{bmatrix}
x_n\\
y_n
\end{bmatrix}, \begin{bmatrix}
x_n\\
y_n
\end{bmatrix}\right\rangle\right|^2
+ \varepsilon^2\displaystyle{\lim_{n\rightarrow \infty}}\left|\left\langle \begin{bmatrix}
0 & \sqrt{\rho(2-\rho)}B\\
0 & (1-\rho)B
\end{bmatrix}\begin{bmatrix}
x_n\\
y_n
\end{bmatrix}, \begin{bmatrix}
x_n\\
y_n
\end{bmatrix}\right\rangle\right|^2
\\&= \displaystyle{\lim_{n\rightarrow \infty}}\left|\left\langle \begin{bmatrix}
0 & \sqrt{\rho(2-\rho)}(A+\varepsilon B)\\
0 & (1-\rho)(A+\varepsilon B)
\end{bmatrix}\begin{bmatrix}
x_n\\
y_n
\end{bmatrix}, \begin{bmatrix}
x_n\\
y_n
\end{bmatrix}\right\rangle\right|^2
\\& =\frac{\rho^2}{4}\,w^2_{\rho}(A+\varepsilon B)\geq \frac{\rho^2}{4} \left(w_{\rho}(A)-\varepsilon^2\right)^2
= \frac{\rho^2}{4}\,w^2_{\rho}(A) - \varepsilon^2\,\frac{\rho^2}{2} w_{\rho}(A)+\varepsilon^4\,\frac{\rho^2}{4},
\end{align*}
\endgroup
which yields
\begingroup\makeatletter\def\f@size{10}\check@mathfonts
\begin{align}\label{I.3.L.002}
&\displaystyle{\lim_{n\rightarrow \infty}}{\rm Re}
\left(\left\langle \sqrt{\rho(2-\rho)}x_n + (1-\rho)y_n, Ay_n\right\rangle
\left\langle By_n, \sqrt{\rho(2-\rho)}x_n + (1-\rho)y_n\right\rangle\right)\nonumber
\\& \qquad \geq \varepsilon^3\,\frac{\rho^2}{8} - \varepsilon\,\frac{\rho^2}{4} w_{\rho}(A)
- \varepsilon\frac{\rho^2}{8}\,w^2_{\rho}(B).
\end{align}
\endgroup
Finally, by letting $\varepsilon\rightarrow 0^+$ in \eqref{I.3.L.002}, we conclude that
\begingroup\makeatletter\def\f@size{10}\check@mathfonts
\begin{align*}
\displaystyle{\lim_{n\rightarrow \infty}}{\rm Re}\left(
\left\langle \sqrt{\rho(2-\rho)}x_n + (1-\rho)y_n, Ay_n\right\rangle
\left\langle By_n, \sqrt{\rho(2-\rho)}x_n + (1-\rho)y_n\right\rangle\right)\geq0.
\end{align*}
\endgroup
(ii)$\Rightarrow$(i) Suppose (ii) holds. For any $r\geq0$, by Lemma \ref{L.001} we have
\begingroup\makeatletter\def\f@size{10}\check@mathfonts
\begin{align*}
&w^2_{\rho}(A+rB) = \frac{4}{\rho^2}\,w\left(\begin{bmatrix}
0 & \sqrt{\rho(2-\rho)}(A+rB) \\
0 & (1-\rho)(A+rB)
\end{bmatrix}\right)
\\&\quad \geq \frac{4}{\rho^2}\,\displaystyle{\lim_{n\rightarrow \infty}}\left|\left\langle \begin{bmatrix}
0 & \sqrt{\rho(2-\rho)}(A+r B)\\
0 & (1-\rho)(A+r B)
\end{bmatrix}\begin{bmatrix}
x_n\\
y_n
\end{bmatrix}, \begin{bmatrix}
x_n\\
y_n
\end{bmatrix}\right\rangle\right|^2
\\& \quad = \displaystyle{\lim_{n\rightarrow \infty}}\left|\left\langle Ay_n, \sqrt{\frac{8-4\rho}{\rho}}x_n+\frac{2-2\rho}{\rho}y_n\right\rangle\right|^2
\\& \qquad + \frac{8}{\rho^2}\,\displaystyle{\lim_{n\rightarrow \infty}}{\rm Re}\left(
\left\langle \sqrt{\rho(2-\rho)}x_n + (1-\rho)y_n, Ay_n\right\rangle
\left\langle By_n, \sqrt{\rho(2-\rho)}x_n + (1-\rho)y_n\right\rangle\right)
\\& \quad \qquad + r^2\displaystyle{\lim_{n\rightarrow \infty}}\left|\left\langle By_n, \sqrt{\frac{8-4\rho}{\rho}}x_n+\frac{2-2\rho}{\rho}y_n\right\rangle\right|^2\geq w^2_{\rho}(A),
\end{align*}
\endgroup
which implies $w_{\rho}(A+rB)\geq w_{\rho}(A)$ and the proof is completed.
\end{proof}
We are now in a position to establish the main result of this section.
\begin{theorem}\label{T.002}
Let $A, B \in \mathbb{B}(\mathcal{H})$ and $\rho\in (0, 2]$. The following conditions are equivalent:
\begin{itemize}
\item[(i)] $A\perp_{w_{\rho}}B$,
\item[(ii)] for each $\theta\in [0, 2\pi)$, there exists a sequence $\left\{\begin{bmatrix}
x_n\\
y_n
\end{bmatrix}\right\}$ in $\mathbf{S}_{\mathcal{H}\oplus \mathcal{H}}$ such that
\begingroup\makeatletter\def\f@size{11}\check@mathfonts
\begin{align*}
\displaystyle{\lim_{n\rightarrow \infty}}\left|\left\langle Ay_n, \sqrt{\frac{8-4\rho}{\rho}}x_n+\frac{2-2\rho}{\rho}y_n\right\rangle\right|=w_{\rho}(A)
\end{align*}
\endgroup
and
\begingroup\makeatletter\def\f@size{10}\check@mathfonts
\begin{align*}
\displaystyle{\lim_{n\rightarrow \infty}}{\rm Re}
\left(e^{i\theta}\left\langle \sqrt{\rho(2-\rho)}x_n + (1-\rho)y_n, Ay_n\right\rangle
\left\langle By_n, \sqrt{\rho(2-\rho)}x_n + (1-\rho)y_n\right\rangle\right)\geq0.
\end{align*}
\endgroup
\end{itemize}
\end{theorem}
\begin{proof}
(i)$\Rightarrow$(ii) Let $A\perp_{w_{\rho}}B$ and let $\theta\in [0, 2\pi)$ be fixed.
Thus, $w_{\rho}(A+re^{i\theta} B)\geq w_{\rho}(A)$ for all $r\geq0$.
By Lemma \ref{L.002} there exists a sequence $\left\{\begin{bmatrix}
x_n\\
y_n
\end{bmatrix}\right\}$ in $\mathbf{S}_{\mathcal{H}\oplus \mathcal{H}}$ such that
\begingroup\makeatletter\def\f@size{11}\check@mathfonts
\begin{align*}
\displaystyle{\lim_{n\rightarrow \infty}}\left|\left\langle Ay_n, \sqrt{\frac{8-4\rho}{\rho}}x_n+\frac{2-2\rho}{\rho}y_n\right\rangle\right|=w_{\rho}(A)
\end{align*}
\endgroup
and
\begingroup\makeatletter\def\f@size{10}\check@mathfonts
\begin{align*}
\displaystyle{\lim_{n\rightarrow \infty}}{\rm Re}\left(\left\langle \sqrt{\rho(2-\rho)}x_n + (1-\rho)y_n, Ay_n\right\rangle
\left\langle e^{i\theta}By_n, \sqrt{\rho(2-\rho)}x_n + (1-\rho)y_n\right\rangle\right)\geq0,
\end{align*}
\endgroup
and hence we deduce (ii).

(ii)$\Rightarrow$(i) Suppose (ii) holds. Let $\gamma\in \mathbb{C}$.
Then $\gamma=e^{i\theta}|\gamma|$ for some $\theta\in [0, 2\pi)$.
So, there exists a sequence $\left\{\begin{bmatrix}
x_n\\
y_n
\end{bmatrix}\right\}$ in $\mathbf{S}_{\mathcal{H}\oplus \mathcal{H}}$ such that
\begingroup\makeatletter\def\f@size{11}\check@mathfonts
\begin{align*}
\displaystyle{\lim_{n\rightarrow \infty}}\left|\left\langle Ay_n, \sqrt{\frac{8-4\rho}{\rho}}x_n+\frac{2-2\rho}{\rho}y_n\right\rangle\right|=w_{\rho}(A)
\end{align*}
\endgroup
and
\begingroup\makeatletter\def\f@size{10}\check@mathfonts
\begin{align*}
\displaystyle{\lim_{n\rightarrow \infty}}{\rm Re}
\left(e^{i\theta}\left\langle \sqrt{\rho(2-\rho)}x_n + (1-\rho)y_n, Ay_n\right\rangle
\left\langle By_n, \sqrt{\rho(2-\rho)}x_n + (1-\rho)y_n\right\rangle\right)\geq0.
\end{align*}
\endgroup
Utilizing a similar argument as in Lemma \ref{L.002}, we get $w_{\rho}(A+|\gamma|e^{i\theta}B)\geq w_{\rho}(A)$.
Thus, $w_{\rho}(A+\gamma B)\geq w_{\rho}(A)$, or equivalently, $A\perp_{w_{\rho}}B$.
\end{proof}
When $\mathcal{H}$ is a finite dimensional Hilbert space, we have the following result.
\begin{corollary}\label{C.002.1}
Let $\mathcal{H}$ be a finite dimensional Hilbert space, and let $A, B \in \mathbb{B}(\mathcal{H})$. For every $\rho\in (0, 2]$, the following conditions are equivalent:
\begin{itemize}
\item[(i)] $A\perp_{w_{\rho}}B$,
\item[(ii)] for each $\theta\in [0, 2\pi)$, there exists a vector $\begin{bmatrix}
x\\
y
\end{bmatrix}\in \mathbf{S}_{\mathcal{H}\oplus \mathcal{H}}$ such that
\begingroup\makeatletter\def\f@size{11}\check@mathfonts
\begin{align*}
\left|\left\langle Ay, \sqrt{\frac{8-4\rho}{\rho}}x+\frac{2-2\rho}{\rho}y\right\rangle\right|=w_{\rho}(A)
\end{align*}
\endgroup
and
\begingroup\makeatletter\def\f@size{10}\check@mathfonts
\begin{align*}
{\rm Re}\left(e^{i\theta}\left\langle \sqrt{\rho(2-\rho)}x + (1-\rho)y, Ay\right\rangle
\left\langle By, \sqrt{\rho(2-\rho)}x + (1-\rho)y\right\rangle\right)\geq0.
\end{align*}
\endgroup
\end{itemize}
\end{corollary}
\begin{proof}
Let $\rho\in (0, 2]$. If $A\perp_{w_{\rho}}B$, then from Theorem \ref{T.002} we obtain a sequence $\left\{\begin{bmatrix}
x_n\\
y_n
\end{bmatrix}\right\}$ in $\mathbf{S}_{\mathcal{H}\oplus \mathcal{H}}$ such that
\begingroup\makeatletter\def\f@size{11}\check@mathfonts
\begin{align*}
\displaystyle{\lim_{n\rightarrow \infty}}\left|\left\langle Ay_n, \sqrt{\frac{8-4\rho}{\rho}}x_n+\frac{2-2\rho}{\rho}y_n\right\rangle\right|=w_{\rho}(A)
\end{align*}
\endgroup
and
\begingroup\makeatletter\def\f@size{10}\check@mathfonts
\begin{align*}
\displaystyle{\lim_{n\rightarrow \infty}}{\rm Re}
\left(e^{i\theta}\left\langle \sqrt{\rho(2-\rho)}x_n + (1-\rho)y_n, Ay_n\right\rangle
\left\langle By_n, \sqrt{\rho(2-\rho)}x_n + (1-\rho)y_n\right\rangle\right)\geq0.
\end{align*}
\endgroup
Since $\left\{\begin{bmatrix}
x_n\\
y_n
\end{bmatrix}\right\}$ is a bounded sequence, then it has a convergent subsequence
converging to a vector $\begin{bmatrix}
x\\
y
\end{bmatrix}$. This $\begin{bmatrix}
x\\
y
\end{bmatrix}$ is the required vector.
The proof in the other direction is straightforward.
\end{proof}
As an immediate consequence of Theorem \ref{T.003}, we have the following result that can also be found in \cite[Theorem~2.3]{Mal.Paul.Sen.MM.2022}.
\begin{corollary}\label{C.002.2}
Let $A, B \in \mathbb{B}(\mathcal{H})$. The following conditions are equivalent:
\begin{itemize}
\item[(i)] $A\perp_{w}B$,
\item[(ii)] for each $\theta\in [0, 2\pi)$, there exists a sequence $\{y_n\}$ in $\mathbf{S}_{\mathcal{H}}$ such that
$\displaystyle{\lim_{n\rightarrow \infty}}\left|\left\langle Ay_n, y_n\right\rangle\right|=w(A)$
and
$\displaystyle{\lim_{n\rightarrow \infty}}{\rm Re}
\left(e^{i\theta}\left\langle y_n, Ay_n\right\rangle
\left\langle By_n, y_n\right\rangle\right)\geq0$.
\end{itemize}
\end{corollary}
\begin{proof}
The proof follows immediately from Theorem \ref{T.003} for $\rho=2$.
\end{proof}
As another consequence of Theorem \ref{T.003} we have the following result.
\begin{corollary}\label{C.002.3}
Let $A, B \in \mathbb{B}(\mathcal{H})$. The following conditions are equivalent:
\begin{itemize}
\item[(i)] $A\perp B$,
\item[(ii)] for each $\theta\in [0, 2\pi)$, there exists a sequence $\left\{\begin{bmatrix}
x_n\\
y_n
\end{bmatrix}\right\}$ in $\mathbf{S}_{\mathcal{H}\oplus \mathcal{H}}$ such that
$\displaystyle{\lim_{n\rightarrow \infty}}\left|\left\langle Ay_n, x_n\right\rangle\right|=\frac{1}{2}\|A\|$
and
$\displaystyle{\lim_{n\rightarrow \infty}}{\rm Re}
\left(e^{i\theta}\left\langle x_n, Ay_n\right\rangle
\left\langle By_n, x_n\right\rangle\right)\geq0$,
\item[(iii)] there exists a sequence $\{z_n\}$ in $\mathbf{S}_{\mathcal{H}}$ such that
$\displaystyle{\lim_{n\rightarrow \infty}}\left\|Az_n\right\|=\|A\|$
and
$\displaystyle{\lim_{n\rightarrow \infty}}\left\langle Az_n, Bz_n\right\rangle=0$.
\end{itemize}
\end{corollary}
\begin{proof}
(i)$\Rightarrow$(ii) This implication follows immediately from Theorem \ref{T.003} for $\rho=1$.

(ii)$\Rightarrow$(iii) Suppose (ii) holds. For $\mu =0, \pi$, there exist, respectively, sequences $\left\{\begin{bmatrix}
x_n\\
y_n
\end{bmatrix}\right\}$ and $\left\{\begin{bmatrix}
x'_n\\
y'_n
\end{bmatrix}\right\}$ in $\mathbf{S}_{\mathcal{H}\oplus \mathcal{H}}$ such that
\begin{align*}
\displaystyle{\lim_{n\rightarrow \infty}}\left|\left\langle Ay_n, x_n\right\rangle\right|=\displaystyle{\lim_{n\rightarrow \infty}}\left|\left\langle Ay'_n, x'_n\right\rangle\right|=\frac{1}{2}\|A\|
\end{align*}
and
\begin{align*}
\displaystyle{\lim_{n\rightarrow \infty}}{\rm Re}\left(
\left\langle x'_n, Ay'_n\right\rangle
\left\langle By'_n, x'_n\right\rangle\right)\leq 0 \leq \displaystyle{\lim_{n\rightarrow \infty}}{\rm Re}\left(
\left\langle x_n, Ay_n\right\rangle
\left\langle By_n, x_n\right\rangle\right).
\end{align*}
Therefore, we can find sequence $\left\{\begin{bmatrix}
x''_n\\
y''_n
\end{bmatrix}\right\}$ in $\mathbf{S}_{\mathcal{H}\oplus \mathcal{H}}$ such that
\begin{align}\label{I.1.C.002.3}
\displaystyle{\lim_{n\rightarrow \infty}}\left|\left\langle Ay''_n, x''_n\right\rangle\right|=\frac{1}{2}\|A\| \quad \mbox{and} \quad
\left\langle By''_n, x''_n\right\rangle=0.
\end{align}
Put $z_n:=2\|x''_n\|y''_n$ and $e_n:=\frac{x''_n}{\|x''_n\|}$. Hence, $\|e_n\|=1$ and
by the arithmetic-geometric mean inequality, we have $\|z_n\|=2\|x''_n\|\|y''_n\|\leq \|x''_n\|^2+\|y''_n\|^2\leq1$, and so $z_n\in\mathbf{S}_{\mathcal{H}}$.
Rewrite \eqref{I.1.C.002.3} as
\begin{align}\label{I.2.C.002.3}
\displaystyle{\lim_{n\rightarrow \infty}}\left|\left\langle Az_n, e_n\right\rangle\right|=\|A\| \quad \mbox{and} \quad
\left\langle e_n, Bz_n\right\rangle=0.
\end{align}
Since $\left|\left\langle Az_n, e_n\right\rangle\right|\leq\|Az_n\|\leq \|A\|$ and ${\left\|\frac{Az_n}{\|Az_n\|}-e_n\right\|}^2=2-\frac{2}{\|Az_n\|}
{\rm Re}\left(\left\langle Az_n, e_n\right\rangle\right)$, by \eqref{I.2.C.002.3} it follows that $\displaystyle{\lim_{n\rightarrow \infty}}\left\|Az_n\right\|=\|A\|$
and $\displaystyle{\lim_{n\rightarrow \infty}}\frac{Az_n}{\|Az_n\|}-e_n=0$.
Therefore, the above facts imply that
\begin{align*}
\displaystyle{\lim_{n\rightarrow \infty}}\left\langle Az_n, Bz_n\right\rangle
= \displaystyle{\lim_{n\rightarrow \infty}}\left(\|Az_n\|\left\langle \frac{Az_n}{\|Az_n\|}-e_n, Bz_n\right\rangle
+\|Az_n\|\left\langle e_n, Bz_n\right\rangle\right)=0.
\end{align*}
(iii)$\Rightarrow$(i) Suppose that there exists a sequence $\{z_n\}$ in $\mathbf{S}_{\mathcal{H}}$ such that
$\displaystyle{\lim_{n\rightarrow \infty}}\left\|Az_n\right\|=\|A\|$
and
$\displaystyle{\lim_{n\rightarrow \infty}}\left\langle Az_n, Bz_n\right\rangle=0$.
We have
\begin{align*}
{\|A + \gamma B\|}^2 \geq {\|(A + \gamma B)z_n\|}^2
= {\|Az_n\|}^2 + 2{\rm Re}(\overline{\gamma}{\langle Az_n, Bz_n\rangle}) + |\gamma|^2{\|Bz_n\|}^2
\end{align*}
for all $\gamma\in\mathbb{C}$ and $n\in\mathbb{N}$. Thus,
\begin{align*}
{\|A + \gamma B\|}^2 \geq \lim_{n\rightarrow\infty}\sup{\|(A + \gamma B)z_n\|}^2 \geq {\|A\|}^2
\end{align*}
for all $\gamma \in \mathbb{C}$, and so $A\perp B$.
\end{proof}
\begin{remark}\label{R.002.4}
The equivalence (i)$\Leftrightarrow$(iii) in Corollary \ref{C.002.3} is due to R.~Bhatia and P.~\v{S}emrl \cite[Remark~3.1]{B.S.LAA.1999}.
\end{remark}
\section{The $w_{\rho}$-parallelism}\label{section.3}
In this section, we explore the $w_{\rho}$-parallelism for Hilbert space operators.
Here are some properties of the $w_{\rho}$-parallelism whose proofs are so easy that we omit them.
\begin{proposition}\label{P.003}
Let $A, B \in \mathbb{B}(\mathcal{H})$ and $\rho\in (0, 2]$. The following conditions are mutually equivalent:
\begin{itemize}
\item[(i)] $A\parallel_{w_{\rho}}B$,
\item[(ii)] $A^*\parallel_{w_{\rho}}B^*$,
\item[(iii)] $\zeta A\parallel_{w_{\rho}}\zeta B$ for all $\zeta\in \mathbb{C}\setminus\{0\}$,
\item[(iii)] $\alpha A\parallel_{w_{\rho}}\beta B$ for all $\alpha, \beta\in \mathbb{R}\setminus\{0\}$,
\item[(iv)] $U^*AU\parallel_{w_{\rho}}U^*BU$ for all unitary $U\in\mathbb{B}(\mathcal{H})$.
\end{itemize}
\end{proposition}
To achieve the following theorem, we mimic some ideas of \cite[Theorem~2.1]{B.B.PAMS.2002}.
\begin{theorem}\label{T.003}
Let $A, B \in \mathbb{B}(\mathcal{H})$ and $\rho\in (0, 2]$. The following conditions are equivalent:
\begin{itemize}
\item[(i)] $A\parallel_{w_{\rho}}B$,
\item[(ii)] there exists a sequence $\left\{\begin{bmatrix}
x_n\\
y_n
\end{bmatrix}\right\}$ in $\mathbf{S}_{\mathcal{H}\oplus \mathcal{H}}$ such that
\begingroup\makeatletter\def\f@size{10}\check@mathfonts
\begin{align*}
\displaystyle{\lim_{n\rightarrow \infty}}\left|\left\langle Ay_n, \sqrt{\frac{8-4\rho}{\rho}}x_n+\frac{2-2\rho}{\rho}y_n\right\rangle\left\langle By_n, \sqrt{\frac{8-4\rho}{\rho}}x_n+\frac{2-2\rho}{\rho}y_n\right\rangle\right|=w_{\rho}(A)w_{\rho}(B).
\end{align*}
\endgroup
\end{itemize}
\end{theorem}
\begin{proof}
(i)$\Rightarrow$(ii) Let $A\parallel_{w_{\rho}}B$. Hence, there exists $\lambda\in\mathbb{T}$ such that
$w_{\rho}(A+\lambda B) = w_{\rho}(A)+w_{\rho}(B)$.
By Lemma \ref{L.001} there exists a sequence $\left\{\begin{bmatrix}
x_n\\
y_n
\end{bmatrix}\right\}$ in $\mathbf{S}_{\mathcal{H}\oplus \mathcal{H}}$ such that
\begingroup\makeatletter\def\f@size{11}\check@mathfonts
\begin{align*}
\displaystyle{\lim_{n\rightarrow \infty}}\left|\left\langle \begin{bmatrix}
0 & \sqrt{\rho(2-\rho)}(A+\lambda B)\\
0 & (1-\rho)(A+\lambda B)
\end{bmatrix}\begin{bmatrix}
x_n\\
y_n
\end{bmatrix}, \begin{bmatrix}
x_n\\
y_n
\end{bmatrix}\right\rangle\right|=\frac{\rho}{2}\,w_{\rho}(A+\lambda B),
\end{align*}
\endgroup
or equivalently,
\begingroup\makeatletter\def\f@size{11}\check@mathfonts
\begin{align}\label{I.1.T.003}
\displaystyle{\lim_{n\rightarrow \infty}}\left|\left\langle (A+\lambda B)y_n, \sqrt{\frac{8-4\rho}{\rho}}x_n+\frac{2-2\rho}{\rho}y_n\right\rangle\right|=w_{\rho}(A+\lambda B).
\end{align}
\endgroup
We have
\begingroup\makeatletter\def\f@size{9}\check@mathfonts
\begin{align*}
&\left|\left\langle (A+\lambda B)y_n, \sqrt{\frac{8-4\rho}{\rho}}x_n+\frac{2-2\rho}{\rho}y_n\right\rangle\right|^2
\\& = \left|\left\langle Ay_n, \sqrt{\frac{8-4\rho}{\rho}}x_n+\frac{2-2\rho}{\rho}y_n\right\rangle\right|^2
+ \left|\left\langle By_n, \sqrt{\frac{8-4\rho}{\rho}}x_n+\frac{2-2\rho}{\rho}y_n\right\rangle\right|^2
\\& \qquad + 2{\rm Re}\left(\overline{\lambda}\left\langle Ay_n, \sqrt{\frac{8-4\rho}{\rho}}x_n+\frac{2-2\rho}{\rho}y_n\right\rangle
\left\langle \sqrt{\frac{8-4\rho}{\rho}}x_n+\frac{2-2\rho}{\rho}y_n, By_n\right\rangle\right)
\\&\leq \frac{4}{\rho^2}\,\left|\left\langle \begin{bmatrix}
0 & \sqrt{\rho(2-\rho)}A\\
0 & (1-\rho)A
\end{bmatrix}\begin{bmatrix}
x_n\\
y_n
\end{bmatrix}, \begin{bmatrix}
x_n\\
y_n
\end{bmatrix}\right\rangle\right|^2
+ \frac{4}{\rho^2}\,\left|\left\langle \begin{bmatrix}
0 & \sqrt{\rho(2-\rho)}B\\
0 & (1-\rho)B
\end{bmatrix}\begin{bmatrix}
x_n\\
y_n
\end{bmatrix}, \begin{bmatrix}
x_n\\
y_n
\end{bmatrix}\right\rangle\right|^2
\\& \qquad + 2\left|\left\langle Ay_n, \sqrt{\frac{8-4\rho}{\rho}}x_n+\frac{2-2\rho}{\rho}y_n\right\rangle
\left\langle \sqrt{\frac{8-4\rho}{\rho}}x_n+\frac{2-2\rho}{\rho}y_n, By_n\right\rangle\right|
\\& \leq w^2_{\rho}(A) + w^2_{\rho}(B)
+ 2\left|\left\langle Ay_n, \sqrt{\frac{8-4\rho}{\rho}}x_n+\frac{2-2\rho}{\rho}y_n\right\rangle
\left\langle By_n, \sqrt{\frac{8-4\rho}{\rho}}x_n+\frac{2-2\rho}{\rho}y_n\right\rangle\right|
\\& \leq w^2_{\rho}(A) + w^2_{\rho}(B)
+ \frac{8}{\rho^2}\,\left|\left\langle \begin{bmatrix}
0 & \sqrt{\rho(2-\rho)}A\\
0 & (1-\rho)A
\end{bmatrix}\begin{bmatrix}
x_n\\
y_n
\end{bmatrix}, \begin{bmatrix}
x_n\\
y_n
\end{bmatrix}\right\rangle\right|
\left|\left\langle \begin{bmatrix}
0 & \sqrt{\rho(2-\rho)}B\\
0 & (1-\rho)B
\end{bmatrix}\begin{bmatrix}
x_n\\
y_n
\end{bmatrix}, \begin{bmatrix}
x_n\\
y_n
\end{bmatrix}\right\rangle\right|
\\& \leq w^2_{\rho}(A) + w^2_{\rho}(B) + 2w_{\rho}(A)w_{\rho}(B),
\end{align*}
\endgroup
and hence
\begingroup\makeatletter\def\f@size{9}\check@mathfonts
\begin{align}\label{I.2.T.003}
&\left|\left\langle (A+\lambda B)y_n, \sqrt{\frac{8-4\rho}{\rho}}x_n+\frac{2-2\rho}{\rho}y_n\right\rangle\right|^2\nonumber
\\&\leq w^2_{\rho}(A) + w^2_{\rho}(B)
+ 2\left|\left\langle Ay_n, \sqrt{\frac{8-4\rho}{\rho}}x_n+\frac{2-2\rho}{\rho}y_n\right\rangle
\left\langle By_n, \sqrt{\frac{8-4\rho}{\rho}}x_n+\frac{2-2\rho}{\rho}y_n\right\rangle\right|\nonumber
\\& \leq w^2_{\rho}(A) + w^2_{\rho}(B) + 2w_{\rho}(A)w_{\rho}(B) = w^2_{\rho}(A+\lambda B).
\end{align}
\endgroup
By \eqref{I.1.T.003} and \eqref{I.2.T.003} we obtain
\begingroup\makeatletter\def\f@size{10}\check@mathfonts
\begin{align*}
\displaystyle{\lim_{n\rightarrow \infty}}\left|\left\langle Ay_n, \sqrt{\frac{8-4\rho}{\rho}}x_n+\frac{2-2\rho}{\rho}y_n\right\rangle\left\langle By_n, \sqrt{\frac{8-4\rho}{\rho}}x_n+\frac{2-2\rho}{\rho}y_n\right\rangle\right|=w_{\rho}(A)w_{\rho}(B).
\end{align*}
\endgroup
(ii)$\Rightarrow$(i) Suppose that there exists a sequence $\left\{\begin{bmatrix}
x_n\\
y_n
\end{bmatrix}\right\}$ in $\mathbf{S}_{\mathcal{H}\oplus \mathcal{H}}$ such that
\begingroup\makeatletter\def\f@size{9}\check@mathfonts
\begin{align}\label{I.3.T.003}
\displaystyle{\lim_{n\rightarrow \infty}}\left|\left\langle Ay_n, \sqrt{\frac{8-4\rho}{\rho}}x_n+\frac{2-2\rho}{\rho}y_n\right\rangle\left\langle By_n, \sqrt{\frac{8-4\rho}{\rho}}x_n+\frac{2-2\rho}{\rho}y_n\right\rangle\right|=w_{\rho}(A)w_{\rho}(B).
\end{align}
\endgroup
Since
\begingroup\makeatletter\def\f@size{9}\check@mathfonts
\begin{align*}
&\left|\left\langle Ay_n, \sqrt{\frac{8-4\rho}{\rho}}x_n+\frac{2-2\rho}{\rho}y_n\right\rangle\left\langle By_n, \sqrt{\frac{8-4\rho}{\rho}}x_n+\frac{2-2\rho}{\rho}y_n\right\rangle\right|
\\& \qquad \leq \left|\left\langle Ay_n, \sqrt{\frac{8-4\rho}{\rho}}x_n+\frac{2-2\rho}{\rho}y_n\right\rangle\right| w_{\rho}(B)
\leq w_{\rho}(A)w_{\rho}(B),
\end{align*}
\endgroup
by \eqref{I.3.T.003} it follows that
\begin{align}\label{I.4.T.003}
\displaystyle{\lim_{n\rightarrow \infty}}\left|\left\langle Ay_n, \sqrt{\frac{8-4\rho}{\rho}}x_n+\frac{2-2\rho}{\rho}y_n\right\rangle\right|=w_{\rho}(A).
\end{align}
By using a similar argument, we also have
\begin{align}\label{I.5.T.003}
\displaystyle{\lim_{n\rightarrow \infty}}\left|\left\langle By_n, \sqrt{\frac{8-4\rho}{\rho}}x_n+\frac{2-2\rho}{\rho}y_n\right\rangle\right|=w_{\rho}(B).
\end{align}
Further, for every $n$, there exist $\lambda_n\in\mathbb{T}$ such that
\begingroup\makeatletter\def\f@size{10}\check@mathfonts
\begin{align*}
&\left|\left\langle Ay_n, \sqrt{\frac{8-4\rho}{\rho}}x_n+\frac{2-2\rho}{\rho}y_n\right\rangle\left\langle By_n, \sqrt{\frac{8-4\rho}{\rho}}x_n+\frac{2-2\rho}{\rho}y_n\right\rangle\right|
\\& \qquad = \lambda_n\left\langle Ay_n, \sqrt{\frac{8-4\rho}{\rho}}x_n+\frac{2-2\rho}{\rho}y_n\right\rangle\left\langle By_n, \sqrt{\frac{8-4\rho}{\rho}}x_n+\frac{2-2\rho}{\rho}y_n\right\rangle.
\end{align*}
\endgroup
From this and \eqref{I.3.T.003} it follows that
\begingroup\makeatletter\def\f@size{9}\check@mathfonts
\begin{align}\label{I.6.T.003}
\displaystyle{\lim_{n\rightarrow \infty}}\lambda_n\left\langle Ay_n, \sqrt{\frac{8-4\rho}{\rho}}x_n+\frac{2-2\rho}{\rho}y_n\right\rangle\left\langle By_n, \sqrt{\frac{8-4\rho}{\rho}}x_n+\frac{2-2\rho}{\rho}y_n\right\rangle=w_{\rho}(A)w_{\rho}(B).
\end{align}
\endgroup
Since $\{\lambda_n\}$ is a bounded sequence in $\mathbb{T}$, there exists a subsequence $\{\lambda_{n_k}\}$ and a
$\lambda \in\mathbb{C}$ such that $\lambda_{n_k} \rightarrow \overline{\lambda}$.
We have
\begin{align*}
\big|1-|\lambda|\big|=\big||\lambda_{n_k}|-|\overline{\lambda}|\big|\leq\big|\lambda_{n_k}-\overline{\lambda}\big| \rightarrow 0,
\end{align*}
and hence $|\lambda|=1$. Thus, $\lambda \in \mathbb{T}$. Since
\begingroup\makeatletter\def\f@size{10}\check@mathfonts
\begin{align*}
&{\rm Re}\left(\overline{\lambda}\left\langle Ay_{n_k}, \sqrt{\frac{8-4\rho}{\rho}}x_{n_k}+\frac{2-2\rho}{\rho}y_{n_k}\right\rangle\left\langle By_{n_k}, \sqrt{\frac{8-4\rho}{\rho}}x_{n_k}+\frac{2-2\rho}{\rho}y_{n_k}\right\rangle\right)
\\&={\rm Re}\left((\overline{\lambda}-\lambda_{n_k})\left\langle Ay_{n_k}, \sqrt{\frac{8-4\rho}{\rho}}x_{n_k}+\frac{2-2\rho}{\rho}y_{n_k}\right\rangle\left\langle By_{n_k}, \sqrt{\frac{8-4\rho}{\rho}}x_{n_k}+\frac{2-2\rho}{\rho}y_{n_k}\right\rangle\right)
\\& \qquad +{\rm Re}\left(\lambda_{n_k}\left\langle Ay_{n_k}, \sqrt{\frac{8-4\rho}{\rho}}x_{n_k}+\frac{2-2\rho}{\rho}y_{n_k}\right\rangle\left\langle By_{n_k}, \sqrt{\frac{8-4\rho}{\rho}}x_{n_k}+\frac{2-2\rho}{\rho}y_{n_k}\right\rangle\right),
\end{align*}
\endgroup
by \eqref{I.6.T.003} we get
\begingroup\makeatletter\def\f@size{9}\check@mathfonts
\begin{align}\label{I.7.T.003}
{\rm Re}\left(\overline{\lambda}\left\langle Ay_{n_k}, \sqrt{\frac{8-4\rho}{\rho}}x_{n_k}+\frac{2-2\rho}{\rho}y_{n_k}\right\rangle\left\langle By_{n_k}, \sqrt{\frac{8-4\rho}{\rho}}x_{n_k}+\frac{2-2\rho}{\rho}y_{n_k}\right\rangle\right)
=w_{\rho}(A)w_{\rho}(B).
\end{align}
\endgroup
For every $k$, by Lemma \ref{L.001} we also have
\begingroup\makeatletter\def\f@size{9}\check@mathfonts
\begin{align*}
w^2_{\rho}(A+\lambda B)&\geq \left|\left\langle (A+\lambda B)y_{n_k}, \sqrt{\frac{8-4\rho}{\rho}}x_{n_k}+\frac{2-2\rho}{\rho}y_{n_k}\right\rangle\right|^2
\\&=\left|\left\langle Ay_{n_k}, \sqrt{\frac{8-4\rho}{\rho}}x_{n_k}+\frac{2-2\rho}{\rho}y_{n_k}\right\rangle\right|^2
+ \left|\left\langle By_{n_k}, \sqrt{\frac{8-4\rho}{\rho}}x_{n_k}+\frac{2-2\rho}{\rho}y_{n_k}\right\rangle\right|^2
\\& \qquad +2{\rm Re}\left(\overline{\lambda}\left\langle Ay_{n_k}, \sqrt{\frac{8-4\rho}{\rho}}x_{n_k}+\frac{2-2\rho}{\rho}y_{n_k}\right\rangle\left\langle By_{n_k}, \sqrt{\frac{8-4\rho}{\rho}}x_{n_k}+\frac{2-2\rho}{\rho}y_{n_k}\right\rangle\right),
\end{align*}
\endgroup
and therefore by \eqref{I.4.T.003}, \eqref{I.5.T.003} and \eqref{I.7.T.003} we conclude that
\begingroup\makeatletter\def\f@size{11}\check@mathfonts
\begin{align*}
w^2_{\rho}(A+\lambda B)\geq w^2_{\rho}(A)+w^2_{\rho}(B)+2w_{\rho}(A)w_{\rho}(B).
\end{align*}
\endgroup
Hence, $w_{\rho}(A+\lambda B)= w_{\rho}(A) + w_{\rho}(B)$, that is, $A\parallel_{w_{\rho}}B$.
\end{proof}
If $\mathcal{H}$ be finite dimensional, then we get a tractable characterization
of the $w_{\rho}$-parallelism as follows.
\begin{corollary}\label{C.003.1}
Let $\mathcal{H}$ be a finite dimensional Hilbert space, and let $A, B \in \mathbb{B}(\mathcal{H})$.
For every $\rho\in (0, 2]$, the following conditions are equivalent:
\begin{itemize}
\item[(i)] $A\parallel_{w_{\rho}}B$,
\item[(ii)] there exists a vector $\begin{bmatrix}
x\\
y
\end{bmatrix}\in \mathbf{S}_{\mathcal{H}\oplus \mathcal{H}}$ such that
\begingroup\makeatletter\def\f@size{10}\check@mathfonts
\begin{align*}
\left|\left\langle Ay, \sqrt{\frac{8-4\rho}{\rho}}x+\frac{2-2\rho}{\rho}y\right\rangle\left\langle By, \sqrt{\frac{8-4\rho}{\rho}}x+\frac{2-2\rho}{\rho}y\right\rangle\right|=w_{\rho}(A)w_{\rho}(B).
\end{align*}
\endgroup
\end{itemize}
\end{corollary}
\begin{proof}
The proof is similar to the proof of Corollary \ref{C.002.1}, only we use Theorem \ref{T.003} instead of Theorem \ref{T.002}.
We omit the details.
\end{proof}
The next result that was proved in \cite[Theorem~2.2]{Meh.Amy.Zam.BIMS.2020} (see also \cite[Proposition~3.6]{Ab.Kit.SM.2015})
follows immediately from Theorem \ref{T.003} for $\rho=2$
\begin{corollary}\label{C.003.2}
Let $A, B \in \mathbb{B}(\mathcal{H})$. The following conditions are equivalent:
\begin{itemize}
\item[(i)] $A\parallel_{w}B$,
\item[(ii)] there exists a sequence $\{y_n\}$ in $\mathbf{S}_{\mathcal{H}}$ such that
\begin{align*}
\displaystyle{\lim_{n\rightarrow \infty}}\left|\left\langle Ay_n, y_n\right\rangle\left\langle By_n, y_n\right\rangle\right|=w(A)w(B).
\end{align*}
\end{itemize}
\end{corollary}
We close this paper with the following result.
\begin{corollary}\label{C.003.3}
Let $A, B \in \mathbb{B}(\mathcal{H})$. The following conditions are equivalent:
\begin{itemize}
\item[(i)] $A\parallel B$,
\item[(ii)] there exists a sequence $\left\{\begin{bmatrix}
x_n\\
y_n
\end{bmatrix}\right\}$ in $\mathbf{S}_{\mathcal{H}\oplus \mathcal{H}}$ such that
$\displaystyle{\lim_{n\rightarrow \infty}}\left|\left\langle Ay_n, x_n\right\rangle
\left\langle By_n, x_n\right\rangle\right|= \frac{1}{4}\|A\|\|B\|$,
\item[(iii)] there exists a sequence $\{z_n\}$ in $\mathbf{S}_{\mathcal{H}}$ such that
$\displaystyle{\lim_{n\rightarrow \infty}}\left|\left\langle Az_n, Bz_n\right\rangle\right|= \|A\|\|B\|$.
\end{itemize}
\end{corollary}
\begin{proof}
(i)$\Rightarrow$(ii) This implication follows immediately from Theorem \ref{T.003} for $\rho=1$.

(ii)$\Rightarrow$(iii) Suppose (ii) holds. Put $z_n:=2\|x_n\|y_n$. Then, by the arithmetic-geometric mean
inequality, we have $\|z_n\|=2\|x_n\|\|y_n\|\leq \|x_n\|^2+\|y_n\|^2\leq1$. Thus, $z_n\in\mathbf{S}_{\mathcal{H}}$.
Also, for $x_n\neq 0$, by the Buzano inequality (see \cite{Buz}) we have
\begin{align*}
4\left|\left\langle Ay_n, x_n\right\rangle\left\langle By_n, x_n\right\rangle\right|&=
\left|\left\langle A(2\|x_n\|y_n), \frac{x_n}{\|x_n\|}\right\rangle\left\langle \frac{x_n}{\|x_n\|}, B(2\|x_n\|y_n)\right\rangle\right|
\\&\leq\frac{\|Az_n\|\|Bz_n\|+\left|\langle Az_n, Bz_n\rangle\right|}{2}
\\&\leq\frac{\|A\|\|B\|+\left|\langle Az_n, Bz_n\rangle\right|}{2}
\\&\leq\frac{\|A\|\|B\|+\|Az_n\|\|Bz_n\|}{2}
\leq\|A\|\|B\|,
\end{align*}
and so
\begin{align}\label{I.1.C.003.3}
4\left|\left\langle Ay_n, x_n\right\rangle\left\langle By_n, x_n\right\rangle\right|\leq\frac{\|A\|\|B\|+\left|\langle Az_n, Bz_n\rangle\right|}{2}
\leq\|A\|\|B\|.
\end{align}
Clearly, \eqref{I.1.C.003.3} holds also when $x_n = 0$.
Now, by letting $n\rightarrow \infty$ in \eqref{I.1.C.003.3}, we obtain $\displaystyle{\lim_{n\rightarrow \infty}}\left|\left\langle Az_n, Bz_n\right\rangle\right|= \|A\|\|B\|$.

(iii)$\Rightarrow$(i) The proof is straightforward and is omitted.
\end{proof}
\begin{remark}\label{R.003.4}
The equivalence (i)$\Leftrightarrow$(iii) of Corollary \ref{C.003.3} already stated in \cite[Theorem~3.3]{Z.M.C.M.B.2015}.
\end{remark}
\textbf{Conflict of interest.} The authors state that there is no conflict of interest.

\textbf{Data availability.} Data sharing not applicable to the present paper as no data sets were generated or analyzed during the current study.
\bibliographystyle{amsplain}

\end{document}